%% file: IPSpaces_3Jul2019_MW.tex
\documentclass[reqno,11pt,a4paper]{amsart}

\usepackage{amsmath,amsthm,amssymb,mathtools}
\usepackage{multicol,multirow,hhline,comment,graphicx,color,url}
\usepackage{mathrsfs}	
\usepackage[margin=1in]{geometry}
\usepackage{accents} 
\usepackage[foot]{amsaddr}
\usepackage{hyperref}
\usepackage{cellspace}		
\usepackage{enumitem} 
\setlist[enumerate,1]{leftmargin=1cm}

\usepackage{accents,stackengine,scalerel} 

\input{ADIPMacros}

\numberwithin{equation}{section}
\numberwithin{figure}{section}
\numberwithin{table}{section}

\allowdisplaybreaks[3]

\begin{document}
\ \vspace{-22pt}

\title{Metrics on sets of interval partitions with diversity}
%
%
%
\author[N.~Forman]{Noah Forman$^1$}
\address{$^1$ Department of Mathematics\\ McMaster University\\ Hamilton, ON L8S 4K1\\ Canada}
\email{noahforman@gmail.com}
\author[S.~Pal]{Soumik Pal$^2$}
\address{$^2$ Department of Mathematics\\ University of Washington\\ Seattle, WA 98195\\ USA}
\email{soumikpal@gmail.com}
\author[D.~Rizzolo]{Douglas Rizzolo$^3$}
\address{$^3$ Department of Mathematics\\ University of Delaware\\ Newark, DE 19716\\ USA}
\email{drizzolo@udel.edu}
\author[M.~Winkel]{Matthias Winkel$^4$}
\address{$^4$ Department of Statistics\\ University of Oxford\\ 24-29 St Giles'\\ Oxford, OX1 3LB\\ UK}
\email{winkel@stats.ox.ac.uk}


\thanks{This research is partially supported by NSF grants {DMS-1204840, DMS-1308340, DMS-1612483, DMS-1855568}, UW-RRF grant A112251, and EPSRC grant EP/K029797/1}
\date{\today}
\begin{abstract}
We first consider interval partitions whose complements are Lebesgue-null and introduce a complete metric that induces the same 
topology as the Hausdorff distance (between complements). This is done using correspondences between intervals. Further 
restricting to interval partitions with $\alpha$-diversity, we then adjust the metric to incorporate diversities. We show that this second metric 
space is Lusin. An important feature of this topology is that path-continuity in this topology implies the continuous evolution of diversities. 
This is important in related work on tree-valued stochastic processes where diversities are branch lengths.
\end{abstract}

\maketitle

\ \vspace{-30pt}

\section{Introduction}
\label{sec:intro}

We define interval partitions following Aldous \cite[Section 17]{AldousExch} and Pitman \cite[Chapter 4]{CSP}.

\begin{definition}\label{def:IP_1}
 An \emph{interval partition} is a set $\beta$ of disjoint, open subintervals of some interval $[0,L]$, that cover $[0,L]$ up to a Lebesgue-null set. We write $\IPmag{\beta}$ to denote $L$. We refer to the elements of $\beta$ as its \emph{blocks}. The Lebesgue measure of a block is called its \emph{mass}.  
\end{definition}

Interval partitions of $[0,1]$ appear naturally as representations of discrete distributions. Indeed, we can order the atoms of a discrete 
distribution and consider intervals whose lengths are the masses of atoms. This is useful e.g.\ to simulate from discrete distributions. 
More generally, an interval partition represents a totally ordered and summable collection of real numbers, for example, the interval partition 
generated naturally by the range of a subordinator (see Pitman and Yor \cite{PitmYor92}), or the partition of $[0,1]$ given by the complement of 
the zero-set of a Brownian bridge (Gnedin and Pitman \cite[Example 3]{GnedPitm05}). They also arise from the so-called stick-breaking schemes; see 
\cite[Example 2]{GnedPitm05}. Furthermore, interval partitions occur as limits of \textit{compositions} of natural numbers $n$, i.e.\ sequences 
of positive integers with sum $n$. Interval partitions serve as extremal points in paintbox representations of composition structures on 
$\mathbb{N}$; see Gnedin \cite{Gnedin97}.  

The set of all interval partitions is denoted by $\HIPspace$.  The subscript $H$ indicates that this set is typically endowed with a metric $d_H$ 
under which the distance between $\beta$ and $\gamma$ is the Hausdorff distance between their complements. Then  
$(\HIPspace,d_H)$ is not complete: some Cauchy sequences such as $\{((i-1)/2^n,i/2^n),1\le i\le 2^n\}\cup\{(1,2)\}$, $n\ge 0$, 
do not converge in $(\HIPspace,d_H)$, since the complement of the ``limiting interval partition'' $\{(1,2)\}$ is not Lebesgue-null. 
Our first aim is to define a complete metric $d_H^\prime$ on $\HIPspace$ that induces the same topology as $d_H$. See Section \ref{sectdefns}.

Of particular interest in the study of interval partitions are random interval partitions formed by arranging the coordinates of an 
$(\alpha,\theta)$-Poisson--Dirichlet distributed random variable in a regenerative random order.  These partitions arise both in the study of random trees and in genetics \cite{Kingman1978,Pitman1995,PitmWink09}. One of the important statistics of these partitions is the continuum analogue of the number of parts of an integer composition, called the diversity (see \cite{GnedinRegenerationSurvey, GPY2006alpha}):

\begin{definition}\label{def:diversity_property} 
  If $0<\alpha <1$, we say that an interval partition $\beta\in\HIPspace$ of an interval $[0,L]$ has the \emph{($\alpha$-)diversity property}, or that $\beta$ is an 
  \emph{interval partition with ($\alpha$-)diversity}, if the following limit exists for every $t\in [0,L]$:
  \begin{equation}
   \IPLT^\alpha_{\beta}(t) := \Gamma(1-\alpha) \lim_{h\downto 0}h^\alpha \#\{(a,b)\in \beta\colon\ |b-a|>h,\ b\leq t\}.\label{eq:IPLT}
  \end{equation}
 \end{definition}
 We denote by $\IPspace\subset\HIPspace$ the set of interval partitions $\beta$ that possess the $\alpha$-diversity property. 

In the context of spinal decompositions of random trees the total diversity of an interval partition corresponds to the length of the spine and 
$\IPLT_{\beta}^\alpha(t)$ for $t\in U\in\beta$ corresponds to the height at which a tree of mass $\Leb(U)$ branches off from the spine. In the context of 
genetic models, each block in the interval partition represents the number of individuals in a population with the same genetic type and the 
total diversity represents the genetic diversity. 

In this paper we will fix $0<\alpha<1$ and suppress it from the notation when doing so will not cause confusion.  In particular, we will use 
$\IPLT_{\beta}(t)$ in place of $\IPLT^\alpha_{\beta}(t)$.  We call $\IPLT_{\beta}(t)$ the \emph{diversity} of the interval partition up to 
$t\in[0,L]$. For $U\in\beta$, $t\in U$, we write $\IPLT_{\beta}(U)=\IPLT_{\beta}(t)$, and we write $\IPLT_{\beta}(\infty) := \IPLT_{\beta}(L)$ to 
denote the \emph{total ($\alpha$-)diversity} of $\beta$.

When studying evolving population models \cite{FengWang07,RuggWalkFava13} or evolving random trees \cite{LohrMytnWint18,Paper4} with connections to Poisson--Dirichlet distributions, it is 
natural to ask whether or not the total diversity evolves continuously.  This provides challenges because $\beta\mapsto \IPLT_\beta(\infty)$ is not 
continuous on $\IPspace$ with respect to the topology induced by $d_H$.  

The second aim of this paper is to introduce a metric $\dI$ on $\IPspace$ which generates the same Borel $\sigma$-algebra as $d_H$ and with respect 
to which the diversity function is continuous.  In fact, we would like more, we would like for $\dI$ to be such that bead crushing constructions of 
random trees as in \cite{PitmWink09} can be used to map continuously evolving interval partitions to continuously evolving trees.  Specifically, we 
let $\mathbb{M}$ be the set of (measure-preserving isometry classes of) compact metric measure spaces with the Gromov--Hausdorff--Prokhorov 
topology.  We would like the map $T\colon\IPspace \to \mathbb{M}$ defined by 
$T(\beta) = ([0,\IPLT_\beta(\infty)], d(\cdot,\cdot) , d\IPLT_\beta^{-1})$ to be 
continuous, where $d(\cdot,\cdot)$ is the standard metric on $\mathbb{R}$ and $d\IPLT_\beta^{-1}=\sum_{U\in\beta}\Leb(U)\delta_{\IPLT_\beta(U)}$ is 
the Stieltjes measure associated with the right inverse of $\IPLT_\beta$. 

The structure of this paper is as follows. We define the metrics on $\HIPspace$ and $\IPspace$, state main results and discuss applications in 
Section \ref{sectdefns}. We provide proofs in Section \ref{sectproofs}. 

\section{Definition of metrics and statement of main results}\label{sectdefns}

Fix $0<\alpha<1$. Our definitions of $d_H^\prime$ and $\dI$ are based on the following notion of correspondences between interval partitions, which is motivated by 
the correspondences that can be used to define the Gromov--Hausdorff metric \cite{EPW06}.
 We adopt the standard discrete mathematics notation $[n] := \{1,2,\ldots,n\}$.

 For $\beta,\gamma\in \HIPspace$, a \emph{correspondence} between $\beta$ and $\gamma$ is a finite sequence of ordered pairs of intervals $(U_1,V_1),\ldots,(U_n,V_n) \in \beta\times\gamma$, $n\geq 0$, where the sequences $(U_j)_{j\in [n]}$ and $(V_j)_{j\in [n]}$ are each strictly increasing in the left-to-right ordering of the interval partitions.

As in the case of the Gromov--Hausdorff metric, we need the notion of the distortion of a correspondence. Specifically
 the $\alpha$-\emph{distortion} of a correspondence $(U_j,V_j)_{j\in [n]}$ between $\beta,\gamma\in\IPspace$, denoted by $\dis_\alpha(\beta,\gamma,(U_j,V_j)_{j\in [n]})$, is defined to be the maximum of the following four 
  quantities:
 \begin{enumerate}[label=(\roman*), ref=(\roman*)]
  \item $\sum_{j\in [n]}|\Leb(U_j)-\Leb(V_j)| + \IPmag{\beta} - \sum_{j\in [n]}\Leb(U_j)$, \label{item:IP_m:mass_1}
  \item $\sum_{j\in [n]}|\Leb(U_j)-\Leb(V_j)| + \IPmag{\gamma} - \sum_{j\in [n]}\Leb(V_j)$, \label{item:IP_m:mass_2}
  \item $\sup_{j\in [n]}|\IPLT_{\beta}(U_j) - \IPLT_{\gamma}(V_j)|$,
  \item $|\IPLT_{\beta}(\infty) - \IPLT_{\gamma}(\infty)|$.
 \end{enumerate}
 Similarly, the \emph{Hausdorff distortion} of a correspondence $(U_j,V_j)_{j\in [n]}$ between $\beta,\gamma\in\HIPspace$, denoted by 
$\dis_H(\beta,\gamma,(U_j,V_j)_{j\in[n]})$, is defined to be the maximum of (i)-(ii).
 
We are now prepared to define $d_H^\prime$ and $\dI$. 
 
 \begin{definition}\label{def:IP:metric}
 For $\beta,\gamma\in\HIPspace$ we define
 \begin{equation}\label{eq:IPH:metric_def}
  d_H^\prime(\beta,\gamma) := \inf_{n\ge 0,\,(U_j,V_j)_{j\in [n]}}\dis_H\big(\beta,\gamma,(U_j,V_j)_{j\in [n]}\big),
 \end{equation}
 where the infimum is over all correspondences from $\beta$ to $\gamma$.

 For $\beta,\gamma\in\IPspace$ we similarly define
 \begin{equation}\label{eq:IP:metric_def}
  \dI(\beta,\gamma) := \inf_{n\ge 0,\,(U_j,V_j)_{j\in [n]}}\dis_\alpha\big(\beta,\gamma,(U_j,V_j)_{j\in [n]}\big).
 \end{equation}
\end{definition}

We will relate $d_H^\prime$ to the Hausdorff metric on compact subsets of $[0,\infty)$. Specifically, when applied to the complements
$C_\beta:=[0,\IPmag{\beta}]\setminus\bigcup_{U\in\beta}U$, the Hausdorff metric gives rise to a metric
$$d_H(\beta,\gamma)=\inf\left\{\varepsilon>0\colon C_\beta\subseteq C_\gamma^\varepsilon\mbox{ and }C_\gamma\subseteq C_\beta^\varepsilon\right\},$$
on $\HIPspace$, where $C^\varepsilon=\{s\in[0,\infty)\colon\inf_{t\in C}|t-s|\le\varepsilon\}$ is the $\varepsilon$-thickening of $C$.

Our main results are as follows.

\begin{theorem}\label{thmdHprime}
  \begin{enumerate}
    \item[\rm(a)] $d_H^\prime\colon\HIPspace^2\to[0,\infty)$ is a metric on $\HIPspace$.
    \item[\rm(b)] $d_H$ and $d_H^\prime$ generate the same separable topology.
    \item[\rm(c)] $(\HIPspace,d_H^\prime)$ is a complete metric space, while $(\HIPspace,d_H)$ is not complete.
    \item[\rm(d)] $\IPspace$ is a Borel subset of $\HIPspace$ that is dense in $\HIPspace$. 
  \end{enumerate}
\end{theorem}

\begin{theorem}\label{thmdI}
  \begin{enumerate}
    \item[\rm(a)] $\dI\colon\IPspace^2\to [0,\infty)$ is a metric on $\IPspace$.
    \item[\rm(b)] The topology on $\IPspace$ generated by $\dI$ is strictly stronger than the subset topology generated by $d_H$ or $d_H^\prime$.   
    \item[\rm(c)] The Borel $\sigma$-algebra generated by $\dI$ equals the one generated by $d_H$ or $d_H^\prime$.
    \item[\rm(d)] $(\IPspace, \dI)$ is Lusin, i.e. homeomorphic to a Borel subset of a compact metric space.
  \end{enumerate}
\end{theorem}

\newcommand{\cSd}{\mathcal{S}^{\downarrow}}

We prove these results in Section \ref{sectproofs}. Before we do so, let us note some of the consequences, which motivated us to
introduce these metrics, and which also demonstrate some further connections to other metrics on interval partitions and related notions. 
Denote by $\mathcal{M}$ the set of compactly supported finite Borel measures on $[0,\infty)$, equipped with the topology of weak convergence, and by
$\cSd=\{(x_k)_{k\ge 1}\colon x_1\ge x_2\ge \cdots\ge 0\mbox{ and }\sum_{k\ge 1}x_k<\infty\}$ the space of summable decreasing sequences
equipped with the $\ell_1$ metric. 

\begin{theorem}\label{thmcont} 
  \begin{enumerate}
    \item[\rm(a)] The map $M\colon\IPspace\!\to\!\mathcal{M}$, $M(\beta)\!=\!\sum_{U\in\beta}\Leb(U)\delta_{\IPLT_\beta(U)}$ is continuous.
    \item[\rm(b)] The diversity map $\beta\mapsto\IPLT_\beta(\infty)$ is $\dI$-continuous on $\IPspace$, but not $\dH$-continuous on $\IPspace$. 
    \item[\rm(c)] The map $\ranked\colon\IPspace\to\cSd$, that associates with $\beta\in\IPspace$ the sequence of decreasing order statistics
      of $(\Leb(U),U\in\beta)$, is continuous.
  \end{enumerate}
\end{theorem}

The proof of (a) follows easily by comparing the $\dI$-metric with the Prokhorov metric
$$d_P(\mu,\nu)=\inf\{\varepsilon>0\colon\mu(C)\le\nu(C^\varepsilon)+\varepsilon\mbox{ for all compact }C\subset[0,\infty)\}$$
Indeed, if $\dI(\beta,\gamma)<\varepsilon$, then there is a correspondence of distortion at most $\varepsilon$. By (i) and (ii), this 
correspondence matches, up to $\varepsilon$, all mass of blocks of $\beta$ and $\gamma$, which $M(\beta)$ and $M(\gamma)$ place onto $[0,\infty)$ 
at locations that, by (iii) are at most $\varepsilon$ apart. Taking into account (iv), this also entails the $\dI$-continuity claimed in 
(b). The continuity claimed in (c) is elementary.

To see that $\dH$-continuity fails in (b), consider any $\beta\in\IPspace$ with continuous $\IPLT_\beta$ and $\IPLT_\beta(\infty)>0$. Let $\beta_n$
be the interval partition obtained from $\beta$ by deleting all but the $n$ longest intervals. Then $d_H(\beta_n,\beta)\rightarrow 0$, but 
$\IPLT_{\beta_n}(\infty)=0$ does not converge to $\IPLT_\beta(\infty)>0$. 

We can combine (a) and (b) by representing $\beta$ as a (single-branch) tree $([0,\IPLT_\beta(\infty)],d,M(\beta))$ in the space $\mathbb{T}$ of 
isometry classes of compact rooted and weighted $\mathbb{R}$-trees equipped with the Gromov--Hausdorff--Prokhorov metric. Here, $d$ is the 
Euclidean metric of $\mathbb{R}$ restricted to $[0,\IPLT_\beta(\infty)]$. With reference to the correspondence definition of this metric in 
\cite[Proposition 6]{M09}, this can be expressed as follows.

\begin{corollary}\label{corT}
  The map $T\colon \IPspace\to \mathbb{T}$ defined by $T(\beta) = ([0,\IPLT_\beta(\infty)], d, M(\beta))$ is continuous.
\end{corollary}

This entails, in particular, that for $\beta(t)$ evolving $\dI$-continuously in $\IPspace$, the associated evolution $T(\beta(t))$ in $\mathbb{T}$
is Gromov--Hausdorff--Prokhorov-continuous. This result when suitably iterated by replacing atoms by further branches (cf.\ the bead-splitting 
constructions of \cite{PitmWink09}) is a key step in our construction of the Aldous diffusion \cite{Paper4} as a $\mathbb{T}$-valued diffusion that
has Aldous's Brownian Continuum Random Tree \cite{AldousCRT1} as its stationary distribution. 

Further key steps towards this goal are certain $\IPspace$-valued diffusions \cite{PartA,PartB,Part2}, which are of independent interest and 
are related to Petrov's \cite{Petrov09} diffusions on spaces of decreasing sequences by a projection via $W$ onto the ranked sequence of  
block masses. In connection with Theorem \ref{thmcont}(b) this entails continuously evolving diversity processes for Petrov's diffusions, which does
not appear to follow from previous constructions \cite{FengWang07,Petrov09,RuggWalk09,Eth14,CdBERS16}. Indeed, other processes have been constructed
by directly modelling a continuously evolving diversity process \cite{RuggWalkFava13}.

\section{Proofs of Theorems \ref{thmdHprime} and \ref{thmdI}}\label{sectproofs}

For the ease of the reader, we will restate all parts of the theorems as propositions/corollaries.

\begin{proposition}\label{prop:IP_metric} 
 The maps $d_H^\prime\colon\HIPspace^2\rightarrow[0,\infty)$ and $\dI\colon \IPspace^2\to [0,\infty)$ are metrics.
\end{proposition}

\begin{proof}
 Symmetry is built into the definition, and we leave positive-definiteness as an exercise for the reader. We will prove that $\dI$ satisfies the 
 triangle inequality. The reader will then easily simplify this proof to obtain the triangular inequality for $d_H^\prime$.
 
 Suppose that $\dI(\eta,\beta) = a$ and $\dI(\beta,\gamma) = b$. Then
 \begin{equation}\label{eq:divbound}
  |\IPLT_{\eta}(\infty)-\IPLT_{\gamma}(\infty)| \leq |\IPLT_{\eta}(\infty)-\IPLT_{\beta}(\infty)| + |\IPLT_{\beta}(\infty)-\IPLT_{\gamma}(\infty)| \leq a+b.
 \end{equation}
 Now take $\e > 0$. It suffices to show that $\dI(\eta,\gamma) \leq a+b+2\e$.
 
 There exist correspondences $(U_j,V_j)_{j\in[m]}$ and $(W_j,X_j)_{j\in[n]}$, from $\eta$ to $\beta$ and from $\beta$ to $\gamma$ respectively, with distortions less than $a+\e$ and $b+\e$ respectively. We will split these two sequences into two parts each. Let $(\hat V_j)_{j\in[k]} = (\hat W_j)_{j\in[k]}$ denote the subsequence of intervals that appear in both $(V_j)_{j\in[m]}$ and $(W_j)_{j\in[n]}$; note that $k$ may equal zero, i.e.\ the overlap may be empty. For each $j\in[k]$, let $\hat U_j$ and $\hat X_j$ denote the intervals in $\eta$ and $\gamma$ respectively that are paired with $\hat V_j = \hat W_j$ in the two correspondences. Then, let $(\hat U_j,\hat V_j)_{j\in [m]\setminus[k]}$ denote the remaining terms in the first correspondence not accounted for in the intersection, and let $(\hat W_j,\hat X_j)_{j\in [n]\setminus[k]}$ denote the remaining terms in the second correspondence. So overall, the sequences $(\hat U_j,\hat V_j)_{j\in[m]}$ and $(\hat V_j,\hat W_j)_{j\in[n]}$ are reorderings of the two correspondences.
 
 We will show that the correspondence $(\hat U_j,\hat X_j)_{j\in[k]}$ has distortion less than $a+b+2\e$. There are four quantities, listed in Definition \ref{def:IP:metric}, that we must bound. Quantity (iv) has already been bounded in (\ref{eq:divbound}). To bound (iii), observe that
 $$
  \sup_{j\in[k]}|(\IPLT_{\eta}(\hat U_j) - \IPLT_{\gamma}(\hat X_j)| \leq \sup_{j\in[k]}\left(|\IPLT_{\eta}(\hat U_j) - \IPLT_{\beta}(\hat V_j)| + |\IPLT_{\beta}(\hat W_j) - \IPLT_{\gamma}(\hat X_j)|\right)
  	< a+b+2\e.
 $$
 We now go about bounding (i), which is more involved. By the triangle inequality,
 \begin{align*}
  &\sum_{j\in[k]}|\Leb(\hat U_j) - \Leb(\hat X_j)| + \IPmag{\eta} - \sum_{j\in[k]}\Leb(\hat U_j)\\
  	&\leq \sum_{j\in[m]}|\Leb(U_j) - \Leb(V_j)| + \sum_{j\in[n]}|\Leb(W_j) - \Leb(X_j)|\\
  	&\quad - \sum_{j\in[m]\setminus[k]}|\Leb(\hat U_j) - \Leb(\hat V_j)|
  	 + \left(\IPmag{\eta} - \sum_{j\in[m]}\Leb(\hat U_j)\right) + \sum_{j\in[m]\setminus[k]}\Leb(\hat U_j).
 \end{align*}
 Since the $(\hat V_j)_{j\in[m]\setminus[k]}$ are members of $\beta$ not listed in $(W_j)_{j\in[n]}$,
 $$
  \sum_{j=k+1}^m\Leb(\hat U_j) - \sum_{j=k+1}^m|\Leb(\hat U_j) - \Leb(\hat V_j)| \leq \sum_{j=k+1}^m\Leb(\hat V_j)
  	\leq \IPmag{\beta} - \sum_{j=1}^n\Leb(W_j),
 $$
 again by the triangle inequality. Thus,
 \begin{align*}
  &\sum_{j\in[k]}|\Leb(\hat U_j) - \Leb(\hat X_j)| + \IPmag{\eta} - \sum_{j\in[k]}\Leb(\hat U_j)\\
  	&\leq \sum_{j\in[m]}\!|\Leb(U_j)\! -\! \Leb(V_j)|\! +\! \IPmag{\eta}\! -\! \sum_{j\in[m]}\Leb(U_j)
  	\!+\! \sum_{j\in[n]}\!|\Leb(W_j)\! -\! \Leb(X_j)|\! +\! \IPmag{\beta}\! -\! \sum_{j\in[n]}\Leb(W_j)\\
  	&\leq a+\e + b+\e.
 \end{align*}
 This is the desired bound on quantity (i) in Definition \ref{def:IP:metric}. The same argument bounds (ii):
 $$
  \sum_{j\in[k]}|\Leb(\hat U_j) - \Leb(\hat X_j)| + \IPmag{\gamma} - \sum_{j\in[k]}\Leb(\hat X_j) \leq a+b+2\e.
 $$
 Therefore $\dI(\eta,\gamma) \leq a+b+2\e$, as desired.
\end{proof}

Note also that $d_H$ is a metric, as it is the pullback of the Hausdorff metric on compact subsets of $[0,\infty)$ under the map
$\beta\mapsto C_\beta=[0,\IPmag{\beta}]\setminus\bigcup_{U\in\beta}U$. 

\begin{proposition}
 $(\HIPspace,d_H)$ is Lusin. Furthermore, $\HIPspace$ is not closed, but a Borel subset of the locally compact space $(\mathcal{C},d_H)$ of (collections of disjoint open intervals that form the complements of) compact subsets of $[0,\infty)$.
\end{proposition}

\begin{proof}
 By \cite[Theorem 7.3.8]{BuraBuraIvan01}, the subspace of compact subsets of $[0,L]$ is compact, hence $(\mathcal{C},d_H)$ is locally compact. $\HIPspace$ is not closed since $\{(j-1)2^{-n},j2^{-n}),j\in[2^n]\}\cup\{(1,2)\}$, $n\ge 1$, is a sequence in $(\HIPspace,d_H)$, even in $(\IPspace,d_H)$, but with Hausdorff limit $\{(1,2)\}$ in $\mathcal{C}\setminus\HIPspace$. The proof that $\HIPspace\subset\mathcal{C}$ is a Borel subset is left to the reader.
\end{proof}


There are some natural operations for interval partitions:


\begin{definition}\label{def:IP:concat}
 We define a scaling map $\scaleI\colon(0,\infty)\times\HIPspace\to\HIPspace$ by setting, for $c>0$ and $\beta\in\HIPspace$
 $$\scaleI[c][\beta]=\{(ca,cb)\colon(a,b)\in\beta\}.$$
 Let $(\beta_a)_{a\in\mathcal{A}}$ denote a family of interval partitions indexed by a totally ordered set $(\cA,\preceq)$. For the purpose of this definition, let 
 $S_{\beta}(a-) := \sum_{b\prec a}\IPmag{\beta_b}$ for $a\in\mathcal{A}$.
 If $S_{\beta}(a-) < \infty$ for every $a\in \mathcal{A}$, then we define the \emph{concatenation}
 \begin{equation}\label{eq:IP:concat_def}
  \Concat_{a\in\mathcal{A}}\beta_a := \{(x+S_{\beta}(a-),y+S_{\beta}(a-)\colon\ a\in\mathcal{A},\ (x,y)\in \beta_a\}.
 \end{equation}
 When $\mathcal{A}=\{a_1,a_2\}$, we denote this by $\beta_{a_1}\concat\beta_{a_2}$. 
 We call $(\beta_a)_{a\in\mathcal{A}}$ \emph{summable} if $\sum_{a\in\mathcal{A}}\IPmag{\beta_a} < \infty$. 
 It is then \emph{strongly summable} if the concatenated partition satisfies the diversity property \eqref{eq:IPLT}. 
%
\end{definition}

It will be useful to separate the diversity of a partition from most of its mass in the following sense. 
For $\eta\in\IPspace$ and $\epsilon>0$, let
\begin{equation}\label{eq:IP:mass_cutoff_for_sep}
 \delta(\eta,\epsilon) := \sup\left\{m>0\colon\ \sum_{U\in\eta}\cf\big\{\Leb(U) < m\big\}\Leb(U) < \epsilon \right\}
\end{equation}
For the purpose of the following, let $A := \{U\in\eta\colon\ \Leb(U) \geq \delta(\eta,\epsilon)\big\}$,
$$S_A(x) := \sum_{(a,b)\in A}(b-a)\cf\{b \leq x\}, \quad \text{and} \quad S_{\eta\setminus A}(x) := \sum_{(a,b)\in\eta\setminus A}(b-a)\cf\{b \leq x\}\quad \text{for }x\ge0.$$
We define
 \begin{equation}\label{eq:IP:separate_LT_mass}
 \begin{split}
   \eta^{\IPLT}_{\epsilon} &:= \big\{ (a-S_A(a),b-S_A(a))\colon (a,b)\in\eta\setminus A\big\}\\
   \text{and}\quad
   \eta^{L}_{\epsilon} &:= \big\{ (a-S_{\eta\setminus A}(a),b-S_{\eta\setminus A}(a))\colon (a,b)\in A\big\}.
 \end{split}
 \end{equation}
Effectively, we form $\eta^L_{\epsilon}$ by taking the large blocks of $\eta$ and sliding them down to sit next to each other, and correspondingly for $\eta^{\IPLT}_{\epsilon}$ with the small blocks. These partitions have the properties
$$
 \IPLT_{\eta^{\IPLT}_{\epsilon}}(\infty) = \IPLT_{\eta}(\infty),\qquad \IPmag{\eta^{\IPLT}_{\epsilon}} \leq \epsilon,\qquad
 \IPLT_{\eta^{L}_{\epsilon}}(\infty) = 0,\qquad \IPmag{\eta^{L}_{\epsilon}} \geq \IPmag{\eta} - \epsilon.
$$

We note the following easy lemma.

\begin{lemma}\label{lem:IP:scale}
 For $c>0$, the scaling functions $\beta\mapsto\scaleI[c][\beta]$ are bijections on $\IPspace$ and on $\HIPspace$; specifically, partitions in the image of $\IPspace$ 
 possess the diversity property with
 \begin{equation}
  \IPLT_{\scaleI[c][\beta]}(ct) = c^\alpha\IPLT_{\beta}(t) \qquad \text{for }\beta\in\IPspace,\ t>0,\ c>0.\label{eq:IP:LT_scaling}
 \end{equation}
 Moreover, for $\beta,\gamma\in\IPspace$,
 \begin{gather}
  \dH'(\beta,\scaleI[c][\beta]) = \left|c-1\right|\IPmag{\beta}, \qquad \dH'(c\scaleI\beta,c\scaleI\gamma) = c\dH'(\beta,\gamma),\label{eq:IP:Haus_scale}\\
  \dI(\beta,\scaleI[c][\beta]) \leq \max\left\{\left|c^\alpha - 1\right|\IPLT_{\beta}(\infty), \left|c-1\right|\IPmag{\beta}\right\},\label{eq:IP:scaling_dist_1}\\
  \text{and}\quad\min\{c,c^\alpha\}\dI(\beta,\gamma) \leq \dI(\scaleI[c][\beta],\scaleI[c][\gamma]) \leq \max\{c,c^\alpha\}\dI(\beta,\gamma).\label{eq:IP:scaling_dist_2}
 \end{gather}\pagebreak[2]
\end{lemma}

\begin{proposition}[$d_H$ is equivalent to $d_H'$ and weaker than $\dI$]\label{prop:Hausdorff}
 \begin{enumerate}[label=(\roman*), ref=(\roman*)]
  \item[\rm(i)] For every $\epsilon > 0$ there exist some $\beta,\gamma\in\IPspace$ for which $d_H(\beta,\gamma) < \epsilon$ and $\dI(\beta,\gamma) > 1/\epsilon$.
  \item[\rm(ii)] For $\beta,\gamma\in\IPspace$, we have $d_H'(\beta,\gamma)\leq \dI(\beta,\gamma)$.
  \item[\rm(iii)] The metrics $d_H$ and $d_H'$ generate the same topology on $\HIPspace$.
 \end{enumerate}
\end{proposition}

The related claim that each of $\dI$, $d_H$, and $d_H'$ generates the same Borel $\sigma$-algebra on $\IPspace$ will be proved 
at the end of this paper, in Proposition \ref{prop:haus:sig}.

\begin{proof}
 (i) Fix $\epsilon > 0$. Consider an arbitrary $\eta\in\IPspace$ with $\IPLT_{\eta}(\infty) > 1/\epsilon$. The pair $\left(\eta,\eta^L_{\epsilon/2}\right)$ defined in \eqref{eq:IP:separate_LT_mass} has the desired property.
 
 (ii) This is immediate from Definition \ref{def:IP:metric} of $d_H'$.
 
 (iii) First, we show $\dH(\beta,\gamma)\leq 3d_H'(\beta,\gamma)$ for every $\beta,\gamma\in\HIPspace$. Suppose $d_H'(\beta,\gamma) < x$ for some $x>0$. Then there is some correspondence $(U_i,V_i)_{i\in [n]}$ from $\beta$ to $\gamma$ with Hausdorff distortion less that $x$. Recall from before Definition \ref{def:IP:metric} that, in a correspondence, the $(U_i)$ and $(V_i)$ are each listed in left-to-right order. Let
 $$\beta' := \Concat_{i\in[n]} \{(0,\Leb(U_i))\}, \qquad \gamma' := \Concat_{i\in[n]} \{(0,\Leb(V_i))\}.$$
 By definition of Hausdorff distortion before Definition \ref{def:IP:metric}, $\IPmag{\beta} - \IPmag{\beta'} < x$, and likewise for $\gamma$ and $\gamma'$. Thus, for each $j\in [n-1]$, the right endpoint of $U_j$ and the left endpoint of $U_{j+1}$ are within distance $x$ of the corresponding point in $\beta'$, and similarly for the left endpoint of $U_1$ and the right endpoint of $U_n$. Thus, $d_H(\beta,\beta') < x$ and correspondingly for $\gamma$. Moreover, by definition of distortion, we also find $d_H(\beta',\gamma') < x$. By the triangle inequality, $\dH(\beta,\gamma)< 3x$, as desired.
 
 Now, consider $\beta\in\HIPspace$ and $\epsilon>0$. Take $\delta_0>0$ small enough that $\sum_{U\in\beta\colon \Leb(U)\leq 2\delta_0}\Leb(U) < \epsilon/3$. Let $K$ denote the number of blocks in $\beta$ with mass at least $2\delta_0$. Take $\delta := \min\{\delta_0,\epsilon/(6K+3)\}$. It suffices to show that for $\gamma\in\HIPspace$, if $d_H(\beta,\gamma)<\delta$ then $d_H'(\beta,\gamma)<\epsilon$.
 
 Suppose $d_H(\beta,\gamma)<\delta$ for some $\gamma\in\HIPspace$. Then for each $U\in\beta$ with $\Leb(U) > 2\delta_0 \geq 2\delta$, the midpoint of $U$ must lie within some block $V$ of $\gamma$. Consider the correspondence from $\beta$ to $\gamma$ that matches each such $(U,V)$. Then, by the bound on $d_H(\beta,\gamma)$, for each such pair, $|\Leb(U)-\Leb(V)|<2\delta\leq \epsilon/3K$. Moreover, by our choice of $\delta_0$, the total mass in $\beta$ excluded from the blocks in the correspondence is at most $\epsilon/3$. Similarly, the reader may confirm that the mass in $\gamma$ excluded from the correspondence is at most $(\epsilon/3)+2K\delta+\delta \leq 2\epsilon/3$. Thus, by Definition \ref{def:IP:metric} of $d_H'$, we have $d_H'(\beta,\gamma)<\epsilon$, as desired.
\end{proof}
 



\begin{lemma}\label{lem:IP:sep}
 $(\IPspace,\dI)$ is path-connected and separable.
\end{lemma}

\begin{proof}
 For path-connectedness, just note that $c\mapsto\scaleI[c][\eta]$, $c\in[0,1]$, is a path from $\emptyset\in\IPspace$ to $\eta\in\IPspace$. Specifically, continuity holds since Lemma \ref{lem:IP:scale} yields for $0\leq a<b\leq 1$
 $$
  \dI(\scaleI[a][\eta],\scaleI[b][\eta]) = \dI\left( \scaleI[\frac{a}{b} ][b\scaleI \eta], \scaleI[b][\eta]\right)
  	\leq \max\left\{\left|b^\alpha - a^\alpha \right|\IPLT_\eta(\infty),\left|b-a\right|\IPmag{\eta}\right\}
. $$
 
 For separability, we fix a partition $\eta\in\IPspace$ with $\IPLT_{\eta}(\infty) > 0$ and such that $t\mapsto\IPLT_\eta(t)$ is continuous on $[0,\IPmag{\eta}]$.
 For the purpose of this proof we abbreviate our scaling notation from $\scaleI[c][\eta]$ to $c\eta$. We will construct a countable $S\subset\IPspace$ in which each element is formed by taking $(c\eta)^{\IPLT}_{\epsilon}$, as in \eqref{eq:IP:separate_LT_mass}, for some $c\geq 0$ and $\epsilon > 0$, and inserting finitely many large blocks into the middle, via the following operation.
%
 For $s\in \left[0,\IPLT_{\eta}(\infty)\right]$ and $m>0$, we define
 \begin{equation*}
  \eta\oplus_s m := \big(\{U\in\eta\colon \IPLT_{\eta}(U) \leq s\} \concat \{(0,m)\}\big) \cup \{(a+m,b+m)\colon (a,b)\in\eta,\ \IPLT_{\eta}(a)>s\}.
 \end{equation*}
 This operation inserts a new interval $V$ of length $m$ into the middle of $\eta$ in such a way that $\IPLT_{\eta\oplus_s m}(V) = s$. Let
 $$
  S := \left\{ (c\eta)^{\IPLT}_{\varepsilon} \oplus_{s_1} m_1 \cdots \oplus_{s_r} m_r\ \middle|
   		\begin{array}{l}
   			r\in\BN,\ s_1,\ldots,s_r \in [0,\IPLT_{c\eta}(\infty))\cap\BQ,\\
   			c,\epsilon,m_1,\ldots,m_r\in (0,\infty)\cap\BQ
   		\end{array}\right\}.
 $$
 By Lemma \ref{lem:IP:scale}, $\IPLT_{c\eta}(\infty) = c^\alpha \IPLT_\eta(\infty)$ for $c \geq 0$. Thus, any $\beta\in\IPspace$ can be approximated in $S$ by the partitions constructed from the following rational sequences. First, take rational
 $$
  c_n \to \left(\frac{\IPLT_\beta(\infty)}{\IPLT_\eta(\infty)}\right)^{1/\alpha},\qquad \epsilon_n = \frac1n \downto 0,\qquad\text{ and}\qquad r_n = \# \beta^L_{\epsilon_n}.
 $$
 Then let $\left\{U\in\beta\colon\ \Leb(U) > \delta\left(\beta,\epsilon_n\right)\right\} = \left\{\left(a_j^{(n)},a_j^{(n)}\!+\!k_j^{(n)}\right),\ j\in [r_n]\right\}$ with $a_1^{(n)}\le\cdots\le a_{r_n}^{(n)}$,
 where $\delta$ is as in \eqref{eq:IP:mass_cutoff_for_sep}. This is the sequence of blocks of $\beta$ that comprise $\beta^L_{\epsilon_n}$. Finally, we take rational sequences $\left(\left(s_j^{(n)},m_j^{(n)}\right),\ j\in [r_n]\right)$ so that 
$$
  \sup\nolimits_{j\in[r_n]}\left|s_j^{(n)}-\IPLT_\beta\left(a_j^{(n)}\right)\right| \le \epsilon_n \quad \text{and} \quad \sum\nolimits_{j\in[r_n]}\left|k_j^{(n)}-m_j^{(n)}\right| \le \epsilon_n.
 $$\vspace{-24pt}
 
\end{proof}

\begin{corollary}\label{cor3}
 There is a metric on $\IPspace$ that generates the same topology as $\dI$, for which $\IPspace$ is isometric to a subset of a compact metric space.
\end{corollary}

\begin{proof}
 Since $(\IPspace,\dI)$ is a separable metric space, Dudley's \cite[Theorem 2.8.2]{Dudley02} applies.
\end{proof}

Unfortunately, this argument is unsuitable to show that the subset can be chosen as a Borel subset. Indeed, the argument can be applied to non-Borel subsets of a compact metric space. To prove this, we introduce a larger metric space 
$(\cJ,d_\cJ)$, on pairs $(\eta,f)$, where $\eta\in\HIPspace$ is an interval partition
and where $f$ is a right-continuous increasing function that is not necessarily $f=\IPLT_\eta(\cdot+)$, which may not even exist, but which shares 
the property of $\IPLT$ to be constant on intervals $U\in\eta$. Then $(\beta,\IPLT_\beta(\cdot+))\in\cJ$ for all $\beta\in\IPspace$. 

The reader may wonder why we take the process of right limits $\IPLT_\beta(\cdot+)$ associated with $\IPLT_\beta$. First note that, in general, 
$\IPLT_\beta$ may be neither left- nor right-continuous. E.g., take any interval partition $\beta$ with positive diversity $D=\IPLT_\beta(\infty)$ 
and reorder the blocks in ranked order of mass. Then the resulting interval partition has zero diversity function, jumping to $D$ at $\IPmag{\beta}$.
If we instead arrange intervals of even rank from the left and of odd rank from the right, accumulating in the ``middle'', at $t$, say, then the 
diversity function of the resulting interval partition is constant $0$ on $(0,t)$, constant $D$ on $(t,\infty)$ and $D/2$ at $t$.  

We use right-continuous functions in $\cJ$ to be definite. We actually only care about the values that $f$ takes on the intervals of constancy. But
we prefer to work with representatives in a familiar class of functions.

\begin{definition}\label{def:dJ_metric} Let $\cJ$ be the set of pairs $(\eta,f)$, where $\eta$ is an interval partition of $[0,\IPmag{\eta}]$ with $\Leb([0,\IPmag{\eta}]\setminus\bigcup_{U\in\eta}U)=0$, and where $f\colon[0,\infty]\rightarrow[0,\infty)$ is a right-continuous increasing function that is constant on every interval 
$U\in\eta$ and on $[\IPmag{\eta},\infty]$. We replace $\IPLT_\eta$ and $\IPLT_\beta$ in Definition \ref{def:IP:metric}, the definition of $\dI(\eta,\beta)$, by $f$ and $g$, to define $d_\cJ((\eta,f),(\beta,g))$. 
\end{definition}

Recall the Skorokhod metric of \cite[equations (14.12), (14.13)]{Billingsley}; we denote this by $d_{\cD}$. For $n\ge1$, let $\cJ_n\subseteq\cJ$ 
denote the set of $(\beta,f)\in\cJ$ for which $\beta$ has exactly $n$ blocks. For $n\geq1$ and $\beta\in\HIPspace$, let $\beta_n$ denote the 
interval partition formed by deleting all but the $n$ largest blocks from $\beta$ (breaking ties via left-to-right order) and sliding these large 
blocks together, as in the construction of $\eta^L_{\epsilon}$ in \eqref{eq:IP:separate_LT_mass}. For $(\beta,f)\in\cJ$, equip $\beta_n$ with the
function $f_n$ that is constant on each block of $\beta_n$ with the value that $f$ takes on the corresponding block of $\beta$. 

\begin{lemma}\label{lem:dJ_metric}
 \begin{enumerate}\item[\rm(i)] The distance function $d_\cJ$ is a metric on $\cJ$.
   \item[\rm(ii)] 
     For $n\geq1$, the metric $\dJ$ on $\cJ_n$ is topologically equivalent to the maximum of $\dH'$ in the first coordinate and $d_{\cD}$ in 
     the second. 
   \item[\rm(iii)] The maps $\beta\mapsto\beta_n$ and $(\beta,f)\mapsto(\beta_n,f_n)$ are Borel under $\dH'$ and $\dJ$ respectively.
   \item[\rm(iv)] The map $(\eta,f)\mapsto f$ is Borel from $(\cJ,\dJ)$ to $(\cD,d_\cD)$.
 \end{enumerate}  
\end{lemma}

\begin{proof}
 (i) Given the proof of Proposition \ref{prop:IP_metric}, the only change needed for this part of the lemma is in proving positive-definiteness, since now $f$ is not determined by $\eta$. However, this follows easily since we assume that $f$ is right-continuous and constant on each $U\in\eta$ and on $[\IPmag{\eta},\infty]$, and $f$ is therefore determined by the values it takes on these sets.
 
 (ii) Fix $(\beta,f)\in\cJ_n$. We denote the blocks of $\beta$ by $U_1,\ldots,U_n$, in left-to-right order. Take $r\in \big(0,\min_{j\in [n]}\Leb(U_j)\big)$. We will show that, for $(\gamma,g)\in\cJ_n$, we get $\dJ((\beta,f),(\gamma,g))<r$ if and only if both $\dH'(\beta,\gamma) < r$ and $d_{\cD}(f,g) < r$.
 
 Consider $(\gamma,g)\in\cJ_n$ with $\dH'(\beta,\gamma) < r$ and $d_{\cD}(f,g) < r$. Since we have required $r$ to be smaller than all block masses in $\beta$, the only correspondence from $\beta$ to $\gamma$ that can have Hausdorff distortion less than $r$ is $(U_i,V_i)_{i\in[n]}$, where $V_1,\ldots,V_n$ denote the blocks of $\gamma$ in left-to-right order. In particular, $\sum_{i\in[n]}|\Leb(V_i)-\Leb(U_i)| < r$. Thus, in order for a continuous time-change $\lambda\colon [0,\IPmag{\beta}]\to [0,\IPmag{\gamma}]$ to never deviate from the identity by $r$, it must map some time in each $U_i$ to a time in the corresponding $V_i$. Therefore, by our bound on $d_{\cD}$, we have $\max_{i\in[n]}|g(V_i)-f(U_i)| < r$. We conclude that $\dJ((\beta,f),(\gamma,g))<r$.
 
 Now, consider $(\gamma,g)\in\cJ_n$ with $\dJ((\beta,f),(\gamma,g)) < r$. Following our earlier notation, the only correspondence that can give distortion less than $r$ is $(U_i,V_i)_{i\in[n]}$. 
It follows immediately from Definitions \ref{def:IP:metric} and \ref{def:dJ_metric} of $\dH'$ and $\dJ$ that $\dH'(\beta,\gamma)\leq \dJ((\beta,f),(\gamma,g)) < r$. We define $\lambda\colon [0,\IPmag{\beta}]\to [0,\IPmag{\gamma}]$ by mapping the left and right endpoints of each $U_j$ to the corresponding left and right endpoints of $V_j$ and interpolating linearly. Since $\sum_{i\in[n]}|\Leb(V_i)-\Leb(U_i)| < r$, it follows that $|\lambda(t)-t| < r$ for $t\in [0,\IPmag{\beta}]$ as well. By definition of $\dJ$, we have $|g(V_i)-f(U_i)| < r$ for each $i\in [n]$. Thus, $|g(\lambda(t)) - f(t)| < r$ for $t\in [0,\IPmag{\beta}]$. This gives $d_{\cD}(f,g) < r$.
 %
 %
 %
 
 (iii) The map $\ranked$ that sends $\beta\in\HIPspace$ to the vector of its order statistics is continuous under $\dH'$. The restriction map $(\beta,t)\mapsto \restrict{\beta}{[0,t]} := \{U\cap (0,t)\colon U\in\beta,\,U\cap (0,t)\neq\emptyset\}$ is continuous from $\dH'$ plus the Euclidean metric to $\dH'$. If $\ranked(\beta) = (x_1,x_2,\ldots)$, then we determine whether the block of mass $x_1$ is to the right of the block of mass $x_2$ by finding the least $t_1,t_2\in x_2\BN$ for which $\restrict{\beta}{[0,t_1]}$ has $x_1$ as its first order statistic and $\restrict{\beta}{[0,t_2]}$ has $(x_1,x_2)$ as its first two order statistics. If $t_1<t_2$ then $\beta_2 = \{(0,x_1),(x_1,x_1+x_2)\}$; otherwise, $\beta_2 = \{(0,x_2),(x_2,x_2+x_1)\}$. This method extends to give the desired measurability of $\beta\mapsto\beta_n$. 
 
 Now let $y_1(\beta,f):=f(U_1)$, where $U_1\in\beta$ is the longest interval (the left-most of these, if there are ties). Then  
 $\{(\gamma,g)\in\cJ\colon y_1(\gamma,g)>z\}$ is open in $(\cJ,\dJ)$. This extends to show the measurability of the functions 
 $y_{n,k}\colon\cJ\rightarrow[0,\infty)$, $1\le k\le n$, that assign to $(\beta,f)$ the values 
 $y_{n,k}(\beta,f)$ of $f$ on the $n$ longest intervals of $\beta$, in left-to-right order. This allows to measurably construct 
 $f_n$ from $(\beta,f)$, which entails the measurability of $(\beta,f)\mapsto(\beta_n,f_n)$.   
 
 (iv) As right-continuity of $f$ entails $\lim_{n\uparrow\infty}f_n(t)=f(t)$, the measurability of $(f,\eta)\mapsto f(t)$ for each $t\in[0,\infty)$ 
 follows from (iii). By \cite[Theorem 14.5]{Billingsley}, the Borel $\sigma$-algebra on Skorokhod space is generated by the evaluation maps, so 
 the claimed measurablity of $(\eta,f)\mapsto f$ follows. 
\end{proof}


 For $t\ge0$, let 
\begin{equation}\label{eq:div_approx}
 D_{\beta,n}(t) := \Gamma(1-\alpha) x_n^\alpha \#\{(a,b)\in\beta_n\colon b\leq t\}, \qquad \text{where} \qquad x_n = \min\{\Leb(U)\colon U\in\beta_n\}.
\end{equation}
If the following two limits are equal, then we adapt Definition \ref{def:diversity_property} to additionally define
\begin{equation}\label{eq:div_rt_cts}
\begin{split}
 \IPLT_{\beta}^+(t) :=&\ \lim_{u\downto t}\limsup_{h\downto 0} \Gamma(1-\alpha) h^\alpha \#\{(a,b)\in\beta\colon (b-a)>h,\, b\leq u\}\\
 	=&\ \lim_{u\downto t}\liminf_{h\downto 0} \Gamma(1-\alpha) h^\alpha \#\{(a,b)\in\beta\colon (b-a)>h,\, b\leq u\}.
\end{split}
\end{equation}

\begin{lemma}\label{lem:dJ_meas}
 \begin{enumerate}[label=(\roman*), ref=(\roman*)]
  \item[\rm(i)] The map $\beta_n\mapsto D_{\beta,n}$ is Borel under $\dH'$.\label{item:dJm:order_stats}
  \item[\rm(ii)] The set $\{(\beta,t)\in \HIPspace\times [0,\infty)\colon \IPLT_{\beta}(t)\text{ exists}\}$ is Borel under $\dH'$ in the first coordinate plus the Euclidean metric in the second. The map $(\beta,t)\mapsto\IPLT_{\beta}(t)$ is measurable on this set, under the same $\sigma$-algebra. The same assertions hold with $\IPLT_{\beta}(t)$ replaced by $\IPLT_{\beta}^+(t)$.
  %
  \label{item:dJm:diversity}
  \item[\rm(iii)] For $\beta\in\IPspace$, the pairs $(\beta_n,D_{\beta,n})$ converge to $(\beta,\IPLT_{\beta}(\,\cdot\,+))$ under $d_{\cJ}$.\label{item:dJm:cnvgc}
 \end{enumerate}
\end{lemma}

\begin{proof}
 (i) The measurability of $\beta_n\mapsto D_{\beta,n}$ follows as in the proof of Lemma \ref{lem:dJ_metric} (iii) from the measurability of $\ranked$ and restrictions.
 
 (ii) The set of interval partitions with finitely many blocks is Borel in $(\HIPspace,d_H')$, and diversity is constant 0 for such interval partitions. It remains to check the claim for interval partitions with infinitely many blocks. Consider $\beta\in\HIPspace$ with infinitely many blocks. For $n\ge 1$, let $U_1,\ldots,U_n$ denote the $n$ largest blocks of $\beta$, in left-to-right order. Let $\theta_{\beta,n}\colon [0,\IPmag{\beta}]\to [0,\IPmag{\beta_n}]$ denote the continuous time-change starting from $\theta_{\beta,n}(0) = 0$, increasing with slope 1 on $\bigcup_{i\in [n]}U_i$, and having slope 0 on $[0,\IPmag{\beta}]\setminus \bigcup_{i\in [n]}\ol U_i$, where $\ol U$ denotes closure. 
 Note that $\{\theta_{\beta,n}(U_1),\ldots,\theta_{\beta,n}(U_n)\} = \beta_n$. 
 It follows from similar arguments to those in the proof of Lemma \ref{lem:dJ_metric} (iii) that $\beta\mapsto\theta_{\beta,n}$ is measurable from $(\HIPspace,\dH')$ to $\cC([0,\infty),[0,\infty))$. Also, $f_n(\beta,t):=D_{\beta,n}(\theta_{\beta,n}(t))$ is Borel since pre-images of $(-\infty,x)$
 are open for all $x\in\mathbb{R}$, i.e.\ $f_n$ is upper semi-continuous.
 
 By comparing \eqref{eq:div_approx} to Definition \ref{def:diversity_property} of $\IPLT_{\beta}$, for every $t\ge0$ we see that $\lim_{n\upto\infty}f_n(\beta,t)\!=\!\IPLT_{\beta}(t)$, with each limit existing if and only if the other exists. As $f_n$ is Borel, this proves the two claims for $\IPLT_{\beta}(t)$. 
 By monotonicity of the limiting terms in \eqref{eq:div_rt_cts}, $\IPLT_{\beta}^+(t)$ exists if and only if
 \begin{equation}\label{eq:divplus}\lim_{m\upto\infty}\limsup_{n\upto\infty}D_{\beta,n}\left(\theta_{\beta,n}\left([2^mt+1]2^{-m}\right)\right) = \lim_{m\upto\infty}\liminf_{n\upto\infty}D_{\beta,n}\left(\theta_{\beta,n}\left([2^mt+1]2^{-m}\right)\right).
 \end{equation}
 If these limits are equal, then they equal $\IPLT_{\beta}^+(t)$. This proves the two claims for $\IPLT_{\beta}^+(t)$.

 
 (iii) This follows from the previous argument by taking the correspondences from $\beta$ to $\beta_n$ that pair $U_i$ with $\theta_{\beta,n}(U_i)$, for each $i\in [n]$.
 %
\end{proof}


\begin{lemma}\label{lmcompl}
 Consider the map $\iota\colon \IPspace\rightarrow\cJ$ given by $\iota(\eta)=(\eta,\IPLT_\eta(\,\cdot\,+))$.
 \begin{enumerate}[label=(\roman*), ref=(\roman*)]
  \item[\rm(i)] Both $\iota(\IPspace)$ and $\cJ\setminus\iota(\IPspace)$ are dense in $(\cJ,d_{\cJ})$.
  \item[\rm(ii)] Both $\iota(\IPspace)$ and $\cJ\setminus\iota(\IPspace)$ are Borel subsets of $\cJ$.
  \item[\rm(iii)] The space $(\cJ,d_\cJ)$ is a completion of $(\IPspace,\dI)$, with respect to the isometric embedding $\iota$.
 \end{enumerate}
\end{lemma}

\begin{proof}
 (i) By the definitions of $\dI$ and $d_\cJ$, the map $\iota$ is an isometry. Take $(\beta,g)\in\cJ\setminus\iota(\IPspace)$ and $\eta\in\IPspace$ with $\IPLT_\eta(\infty) = g(\infty)\ge 0$ and such that $t\mapsto\IPLT_\eta(t)$ is continuous on $[0,\IPmag{\eta}]$. Using the notation of the proof of Lemma \ref{lem:IP:sep}, we consider
  $$
   \beta^{(n)} := \eta_{1/n}^{\IPLT}\oplus_{g\left(a_1^{(n)}\right)}k_1^{(n)}\cdots\oplus_{g\left(a_{r_n}^{(n)}\right)}k_{r_n}^{(n)}.
  $$
  Then $d_\cJ\left(\left(\beta^{(n)},\IPLT_{\beta^{(n)}}\right),(\beta,g)\right)\rightarrow 0$, i.e. $\beta$ is in the closure of $\iota(\IPspace)$. The same argument, with roles of $(\eta,\IPLT_\eta(\,\cdot\,+))$ and $(\beta,g)$ swapped (now $\eta\in\IPspace$ general and $(\beta,g)\in\cJ\setminus\iota(\IPspace)$ and such that $g$ is continuous), shows that $(\eta,\IPLT_\eta)$ is in the closure of $\cJ\setminus\iota(\IPspace)$. 
 
 
 (ii) Recall that for $\eta\in\IPspace$ we have $\IPLT_{\eta}(\,\cdot\,+)=\IPLT_{\eta}^+$ identically. Thus,
 $$\iota(\IPspace)=\left\{(\eta,f)\in\mathcal{J}\colon\mbox{for all }t\in[0,\IPmag{\eta}],\; \IPLT_\eta(t)\mbox{ exists and }\IPLT^+_\eta(t)=f(t)\right\}.$$
 By Lemmas \ref{lem:dJ_metric} (iv) and \ref{lem:dJ_meas} (ii)
, the following set is Borel under $\dJ$:
 $$A := \left\{(\eta,f)\in\cJ\colon \mbox{for all }t\in[0,\IPmag{\eta}]\cap\BQ,\; \IPLT^+_{\eta}(t)\text{ exists and equals }f(t)\right\}.$$
 For $(\eta,f)\in A$, writing $\IPLT^+_\eta(t)$ as in \eqref{eq:divplus}, we find that $\IPLT^+_\eta(t)$ exists for all $t\in[0,\infty)$ and 
 by the right-continuity and monotonicity of $f$ and $\IPLT_{\eta}^+$ 
 we have $f = \IPLT_{\eta}^+$ identically. 
 By comparing Definition \ref{def:diversity_property} of $\IPLT_{\eta}$ with \eqref{eq:div_rt_cts}, we see that if $\IPLT_{\eta}^+$ is continuous at some $t\in [0,\IPmag{\eta}]$ then $\IPLT_{\eta}(t)$ exists and equals $\IPLT_{\eta}^+(t)$, by a sandwiching argument. Thus, $\iota(\IPspace)$ is the set of $(\eta,f)\in A$ for which $\IPLT_{\eta}(t)$ exists at each time $t$ at which $f$ jumps.
 
 By Lemma \ref{lem:dJ_metric} (iv), $(\eta,f)\mapsto f$ is measurable from $(A,\dJ)$ into Skorokhod space.  
 By \cite[Proposition II.(1.16)]{JacodShiryaev}, the map from $f$ to the point process of its jumps is measurable; and by \cite[Proposition 9.1.XII]{DaleyVereJones2}, we can measurably map the latter to a sequence $(t_1,\Delta_1),(t_2,\Delta_2),\ldots$ listing times and sizes of all jumps of $f$, though these may not be listed in chronological order. We write $\tau_i(\eta,f) := t_i$, or $\tau_i(\eta,f) := -1$ if $f$ has less than $i$ jumps. Then
 \begin{equation*}
  \iota(\IPspace) = \left\{(\eta,f)\in A\colon \text{for all }i\in\BN,\;\tau_i(\eta,f) = -1\text{ or }\IPLT_\eta(\tau_i(\eta,f))\text{ exists}\right\}.
 \end{equation*}
 By Lemma \ref{lem:dJ_meas} (ii)
, this set is measurable.
 
 (iii) It is clear from the definition of $\dJ$, based on that of $\dI$, that $\iota$ is an isometry. Now consider any Cauchy sequence $((\eta_n,f_n),\ n\geq 1)$ in $(\cJ,\d_{\cJ})$. Then $(f_n(\infty),n\ge 1)$ is a Cauchy sequence in $[0,\infty)$; let us denote the limit by $f(\infty)$. Consider $(\beta^{(0)},f^{(0)})=(\emptyset,f(\infty))\in\cJ$, i.e. the empty partition with the increasing function that is constant $f(\infty)$. Recall the space $\cSd$ introduced above Theorem \ref{thmcont}. 
 Let $\fs_n=(s_n^{(i)},i\ge 1)=(\Leb (U),U\in\eta_n)^\downarrow\in\cSd$ be the decreasing rearrangement of interval sizes. Then for all correspondences $(U_j,V_j)_{j\in[k]}$,
 $$
  \ell^1(\fs_n,\fs_m)=\sum_{i\ge 1}^\infty\left|s_n^{(i)}-s_m^{(i)}\right|
                      \le\sum_{j\in[k]}\left|\Leb (U_j)-\Leb (V_j)\right|+\IPmag{\eta_n}+\IPmag{\eta_m} - \sum_{j\in[k]}\Leb(U_j)+\Leb(V_j),
 $$
 Let $\epsilon>0$. By the Cauchy property of $((\eta_n,f_n),\ n\geq 1)$, there is some $N_1\ge 1$ so that $d_\cJ\big((\eta_n,f_n),(\eta_m,f_m)\big)<\epsilon/2$ for all $m,n\ge N_1$. Taking the infimum over all correspondences on the RHS of the display, this yields $\ell^1(\fs_n,\fs_m) < \epsilon$ for all $m,n\ge N_1$. By completeness of $(\cSd,\ell^1)$, we have convergence $\fs_n \rightarrow \fs = (s^{(i)},i\ge 1)\in\cSd$.
 
 Now consider any $r\ge 1$ such that $s^{(r)} > s^{(r+1)}$. Consider $\epsilon>0$ with $3\epsilon < s^{(r)}-s^{(r+1)}$. Then there is $N_2\ge 1$ such that for all $n\ge N_2$, there are precisely $r$ intervals $(a_1^{(n)},a_1^{(n)}+k_1^{(n)}),$ $\ldots,(a_r^{(n)},a_r^{(n)}+k_r^{(n)})\in\eta_n$ of length greater than $s^{(r)}-\epsilon$. We define
 $
  \beta_n^{(r)} := \Concat_{j\in[r]} \left\{\left(0,k_j^{(n)}\right)\right\}
 $
 and associate to these intervals the $f_n$-values of the corresponding intervals in $\eta_n$:
 $$
  f_n^{(r)}\left( k_1^{(n)} + \cdots + k_{j}^{(n)} + x \right)=\left\{\begin{array}{ll}f_n\left(a_j^{(n)}\right)& \text{for }x\in [0,k_j^{(n)}),\ j\in [0,r-1],\\
  f_n(\infty)&\text{for }x\geq k_1^{(n)} + \cdots + k_{r}^{(n)}.
  \end{array}\right.
 $$
 Then $d_\cJ((\beta_n^{(r)},f_n^{(r)}),(\beta_m^{(r)},f_m^{(r)}))\le d_\cJ((\eta_n,f_n),(\eta_m,f_m))$, so $(\beta_n^{(r)},f_n^{(r)})$, $n\ge 1$, is a Cauchy sequence in $(\cJ,d_\cJ)$; and since for $n\ge N_2$
 $$
  d_\cJ(\beta_n^{(r)},\beta_m^{(r)}) = \max\left\{ \sup\nolimits_{j\in[r]}\left|f_n(a_j^{(n)})-f_m(a_j^{(m)})\right| ,\ \sum\nolimits_{j\in[r]}\left|k_j^{(n)}-k_j^{(m)}\right| \right\},
 $$
 the vector $\left(\left(f_n\left(a_j^{(n)}\right),k_j^{(n)}\right),\ 1\le j\le r\right)$ is a Cauchy sequence in the metric space $(\BR^{2r},\|\cdot\|_\infty)$. By completeness of $(\BR^{2r},\|\cdot\|_\infty)$, we have convergence to a limit $((f_j,k_j),1\le j\le r)$, which gives rise to a $d_\cJ$-limit $(\beta^{(r)}, f^{(r)})\in\cJ$ of $((\beta_n^{(r)},f_n^{(r)}),n\ge 1)$. By construction, $(\beta^{(r)},f^{(r)})$ is consistent as $r$ varies, in the sense that they are related by insertions of intervals of sizes from $\fs$, and natural correspondences demonstrate that convergence $(\beta^{(r)},f^{(r)})\rightarrow(\beta,f)$ holds in $\cJ$ for a limiting $(\beta,f)\in\cJ$ that incorporates intervals of all sizes $s^{(i)}$, $i\ge 1$.
 
 Finally, let $\epsilon>0$ and $N_1\ge 1$ be as above. Then there is $r$ large enough so that, following the notation of \eqref{eq:IP:separate_LT_mass},
 $$\IPmag{\beta^{\IPLT}_{s^{(r)}}} = \sum_{j\ge r+1}s_j < \epsilon/4.$$
 Since $\fs_n\rightarrow\fs$, there is $N_3\ge N_1$ such that for all $n\ge N_3$, we have $\ell^1(\fs_n,\fs)<\epsilon/4$. Finally, there is $N_4\ge N_3$ so that for all $n\ge N_4$ we have $d_\cJ\left(\left(\beta^{(r)}_n,f_n^{(r)}\right),\ \left(\beta^{(r)},f^{(r)}\right)\right) < \epsilon/4.$ Then for all $n\ge N_4$, we have
 \begin{align*}
  d_\cJ((\eta_n,\!f_n),(\beta,\!f)) &\leq d_\cJ((\eta_n,\!f_n),(\beta_n^{(r)}\!,\!f_n^{(r)}))+d_\cJ((\beta^{(r)}_n\!,\!f_n^{(r)}),(\beta^{(r)}\!,\!f^{(r)}))+d_\cJ((\beta^{(r)}\!,\!f^{(r)}),(\beta,\!f))\\
  &< \sum_{j=r+1}^\infty s_j+\ell^1(\fs_n,\fs)+\frac{\epsilon}{4}+\frac{\epsilon}{4}<\epsilon.
 \end{align*}
 Hence, $((\eta_n,f_n),\ n\ge 1)$ converges to $(\beta,f)$ in $(\cJ,d_\cJ)$. Therefore, $(\cJ,d_\cJ)$ is complete.
\end{proof}

\begin{corollary}\label{corcomplsep}
 $(\cJ,d_\cJ)$ and $(\HIPspace,d_H^\prime)$ are complete and separable metric spaces.
\end{corollary}
\begin{proof} We have shown in the lemmas that $(\cJ,d_\cJ)$ is a complete metric space, and since the completion of a separable metric space is 
also separable, $(\cJ,d_\cJ)$ is also separable. As $(\HIPspace,d_H^\prime)$ has a natural isometrical embedding $\HIPspace\times\{0\}\subset\cJ$
and $\HIPspace\times\{0\}$ is closed in $(\cJ,d_\cJ)$, completeness and separability of $(\HIPspace,d_H^\prime)$ follow.
\end{proof}

\begin{corollary} The sets $\IPspace$ and $\HIPspace\setminus\IPspace$ are dense Borel subsets of $(\HIPspace,d_H^\prime)$.
\end{corollary}

\begin{lemma}
 The space $(\cJ,d_\cJ)$ is not locally compact.
\end{lemma}
\begin{proof}
 Consider the interval partitions $\eta_n=\{((k-1)2^{-n},k2^{-n}),1\le k\le 2^n\}$ and $f_n\equiv 0$. For $m<n$, any correspondence for $\eta_m$ and $\eta_n$ that matches up any intervals of $\eta_m$ and $\eta_n$ attracts a term $2^{-m}-2^{-n}\ge 2^{-n}$, so it is best to use the trivial correspondence which gives $d_\cJ((\eta_n,f_n),(\eta_m,f_m))=1$. Now assume that $(\emptyset,0)\in\cJ$ has a compact neighbourhood $K$. Then $K$ contains an open ball of some radius $2\epsilon>0$, which contains $(\scaleI[\epsilon][\eta_n],0)$ for all $n\ge 1$. Covering $K$ with open balls of radius $\epsilon/2$, the open balls around $(\scaleI[\epsilon][\eta_n],0)$ are disjoint, so there cannot be a finite subcover. This contradicts the compactness of $K$. Hence $(\emptyset,0)$ does not have a compact neighbourhood, and $(\cJ,d_\cJ)$ is not locally compact. 
\end{proof}

Even though $(\cJ,d_\cJ)$ is not locally compact, we can now deduce that $(\IPspace,\dI)$ is Lusin:

\begin{proposition}\label{prop:Lusin2}
 The metric space $(\IPspace,\dI)$ is isometric to a path-connected Borel subset of a complete separable metric space $(\cJ,d_\cJ)$. Furthermore,
 $(\IPspace,\dI)$ is Lusin. 
\end{proposition}

\begin{proof}
Lemma \ref{lmcompl} and Corollary \ref{corcomplsep} yield that 
$(\IPspace,\dI)$ is isometric to a Borel subset of the Polish space $(\cJ,d_\cJ)$. By \cite[Theorem II.82.5]{RogersWilliams}, this implies $(\IPspace,\dI)$ is Lusin.
\end{proof}



\begin{proposition}\label{prop:haus:sig} Each of $\dI$, $d_H$, and $d_H'$ generate the same Borel $\sigma$-algebra on $\IPspace$. 
\end{proposition}

\begin{proof}
 In light of Proposition \ref{prop:Hausdorff} (ii)--(iii), we need only check that all $\dI$-balls are Borel sets with respect to $d_H'$. Recall the notation $\beta_n$ and $D_{\beta,n}$ of introduced before Lemmas \ref{lem:dJ_metric} and \ref{lem:dJ_meas}. By Lemma \ref{lem:dJ_meas} (iii)
, the $\dI$-ball of radius $r>0$ about $\beta$ equals
 $$\bigcup_{m\ge1}\bigcup_{N\ge1}\bigcap_{n>N}\left\{\gamma\in\IPspace\colon d_{\cJ}\big((\beta_n,D_{\beta,n}),(\gamma_n,D_{\gamma,n})\big) < r-m^{-1}\right\}.$$
 The claimed measurability now follows by Lemmas \ref{lem:dJ_meas} (i) 
 and \ref{lem:dJ_metric} (ii).
\end{proof}


\bibliographystyle{abbrv}
\bibliography{AldousDiffusion}
\end{document}

%% file: ADIPMacros.tex
\theoremstyle{plain}

\newtheorem{theorem}{Theorem}[section]

\newtheorem{proposition}[theorem]{Proposition}
\newtheorem{corollary}[theorem]{Corollary}
\newtheorem{lemma}[theorem]{Lemma}

\theoremstyle{definition}
\newtheorem{definition}[theorem]{Definition}

\theoremstyle{remark}

\newcommand{\ranked}{\textsc{ranked}}

\makeatletter
 \DeclareRobustCommand{\checkarg}{\@ifnextchar[{\@witharg}{}}
 \DeclareRobustCommand{\@witharg}[1][]{\ensuremath{\left(#1\right)}}
 
 \DeclareRobustCommand{\scaleGen}[1]{\@ifnextchar[{\@scalewithargs{#1}}{\odot^{}_{#1}}}
 \def\@scalewithargs#1[#2][#3]{#2 \odot^{}_{#1} #3}
\makeatother

\def\IPspace{\mathcal{I}_\alpha}
\def\dI{d_{\alpha}}
\def\HIPspace{\mathcal{I}_H}

\def\dH{d_H}

\def\dJ{d_{\cJ}}

\def\cJ{\mathcal{J}}
\def\dJ{d_{\cJ}}

\def\fs{\mathbf{s}}

\def\cC{\mathcal{C}}	

\def\cD{\mathcal{D}}

\def\dis{\textnormal{dis}}			
\def\IPmag#1{\left\|\vphantom{I}#1\right\|}		

\def\scaleI{\scaleGen{\textnormal{}}}	

\def\restrict#1#2{#1|_{#2}}

\def\Concat{ \mathop{ \raisebox{-2pt}{\Huge$\star$} } }

\def\concat{\star}

\newcommand{\IPLT}{\mathscr{D}}

\newcommand{\ol}[1]{\widebar{#1}}

\def\BR{\mathbb{R}}				
\def\BN{\mathbb{N}}				
\def\BQ{\mathbb{Q}}				
\def\Leb{\textnormal{Leb}}		
\def\e{\epsilon}
\def\d{\delta}
\def\to{\rightarrow}
\def\downto{\downarrow}
\def\upto{\uparrow}

\def\cf{\mathbf{1}}				
\def\BQ{\mathbb{Q}}				

\def\cA{\mathcal{A}}	

\stackMath

\makeatletter
\let\save@mathaccent\mathaccent
\newcommand*\if@single[3]{%
  \setbox0\hbox{${\mathaccent"0362{#1}}^H$}%
  \setbox2\hbox{${\mathaccent"0362{\kern0pt#1}}^H$}%
  \ifdim\ht0=\ht2 #3\else #2\fi
  }
\newcommand*\rel@kern[1]{\kern#1\dimexpr\macc@kerna}
\newcommand{\widebar}{}
\DeclareRobustCommand*\widebar[1]{\@ifnextchar^{\wide@bar{#1}{0}}{\wide@bar{#1}{1}}}
\newcommand*\wide@bar[2]{\if@single{#1}{\wide@bar@{#1}{#2}{1}}{\wide@bar@{#1}{#2}{2}}}
\newcommand*\wide@bar@[3]{%
  \begingroup
  \def\mathaccent##1##2{%
    \let\mathaccent\save@mathaccent
    \if#32 \let\macc@nucleus\first@char \fi
    \setbox\z@\hbox{$\macc@style{\macc@nucleus}_{}$}%
    \setbox\tw@\hbox{$\macc@style{\macc@nucleus}{}_{}$}%
    \dimen@\wd\tw@
    \advance\dimen@-\wd\z@
    \divide\dimen@ 3
    \@tempdima\wd\tw@
    \advance\@tempdima-\scriptspace
    \divide\@tempdima 10
    \advance\dimen@-\@tempdima
    \ifdim\dimen@>\z@ \dimen@0pt\fi
    \rel@kern{0.6}\kern-\dimen@
    \if#31
      \overline{\rel@kern{-0.6}\kern\dimen@\macc@nucleus\rel@kern{0.4}\kern\dimen@}%
      \advance\dimen@0.4\dimexpr\macc@kerna
      \let\final@kern#2%
      \ifdim\dimen@<\z@ \let\final@kern1\fi
      \if\final@kern1 \kern-\dimen@\fi
    \else
      \overline{\rel@kern{-0.6}\kern\dimen@#1}%
    \fi
  }%
  \macc@depth\@ne
  \let\math@bgroup\@empty \let\math@egroup\macc@set@skewchar
  \mathsurround\z@ \frozen@everymath{\mathgroup\macc@group\relax}%
  \macc@set@skewchar\relax
  \let\mathaccentV\macc@nested@a
  \if#31
    \macc@nested@a\relax111{#1}%
  \else
    \def\gobble@till@marker##1\endmarker{}%
    \futurelet\first@char\gobble@till@marker#1\endmarker
    \ifcat\noexpand\first@char A\else
      \def\first@char{}%
    \fi
    \macc@nested@a\relax111{\first@char}%
  \fi
  \endgroup
}
\makeatother

%% file: IPSpaces_3Jul2019_MW.bbl
\def\polhk#1{\setbox0=\hbox{#1}{\ooalign{\hidewidth
  \lower1.5ex\hbox{`}\hidewidth\crcr\unhbox0}}}
\begin{thebibliography}{10}

\bibitem{AldousCRT1}
D.~Aldous.
\newblock The continuum random tree. {I}.
\newblock {\em Ann. Probab.}, 19(1):1--28, 1991.

\bibitem{AldousExch}
D.~J. Aldous.
\newblock Exchangeability and related topics.
\newblock In {\em \'{E}cole d'\'et\'e de probabilit\'es de {S}aint-{F}lour,
  {XIII}---1983}, volume 1117 of {\em Lecture Notes in Math.}, pages 1--198.
  Springer, Berlin, 1985.

\bibitem{Billingsley}
P.~Billingsley.
\newblock {\em Convergence of probability measures}.
\newblock Wiley Series in Probability and Statistics: Probability and
  Statistics. John Wiley \& Sons, Inc., New York, second edition, 1999.
\newblock A Wiley-Interscience Publication.

\bibitem{BuraBuraIvan01}
D.~Burago, Y.~Burago, and S.~Ivanov.
\newblock {\em A course in metric geometry}, volume~33 of {\em Graduate Studies
  in Mathematics}.
\newblock American Mathematical Society, Providence, RI, 2001.

\bibitem{CdBERS16}
C.~Costantini, P.~De~Blasi, S.~N. Ethier, M.~Ruggiero, and D.~Spanò.
\newblock Wright--{F}isher construction of the two-parameter
  {P}oisson--{D}irichlet diffusion.
\newblock {\em Ann. Appl. Probab.}, 27(3):1923--1950, 06 2017.

\bibitem{DaleyVereJones2}
D.~J. Daley and D.~Vere-Jones.
\newblock {\em An introduction to the theory of point processes. {V}ol. {II}}.
\newblock Probability and its Applications (New York). Springer, New York,
  second edition, 2008.
\newblock General theory and structure.

\bibitem{Dudley02}
R.~M. Dudley.
\newblock {\em Real analysis and probability}, volume~74 of {\em Cambridge
  Studies in Advanced Mathematics}.
\newblock Cambridge University Press, Cambridge, 2002.
\newblock Revised reprint of the 1989 original.

\bibitem{Eth14}
S.~Ethier.
\newblock A property of {P}etrov's diffusion.
\newblock {\em Electron. Commun. Probab.}, 19:no. 65, 1--4, 2014.

\bibitem{EPW06}
S.~N. Evans, J.~Pitman, and A.~Winter.
\newblock Rayleigh processes, real trees, and root growth with re-grafting.
\newblock {\em Probability Theory and Related Fields}, 134(1):81--126, 2006.

\bibitem{FengWang07}
S.~Feng and F.-Y. Wang.
\newblock A class of infinite-dimensional diffusion processes with connection
  to population genetics.
\newblock {\em J. Appl. Probab.}, 44(4):938--949, 2007.

\bibitem{Paper4}
N.~Forman, S.~Pal, D.~Rizzolo, and M.~Winkel.
\newblock {Aldous diffusion I: a projective system of continuum $k$-tree
  evolutions}.
\newblock arXiv: 1809.07756v1 [math.PR], 2018.

\bibitem{PartA}
N.~Forman, S.~Pal, D.~Rizzolo, and M.~Winkel.
\newblock {Diffusions on a space of interval partitions: construction from
  marked L{\'e}vy processes}.
\newblock Work in progress, revising parts of arXiv: 1609.06706v2 [math.PR],
  2019.

\bibitem{PartB}
N.~Forman, S.~Pal, D.~Rizzolo, and M.~Winkel.
\newblock {Diffusions on a space of interval partitions: Poisson--Dirichlet
  stationary distributions}.
\newblock Work in progress, revising parts of arXiv: 1609.06706v2 [math.PR],
  2019.

\bibitem{Part2}
N.~Forman, S.~Pal, D.~Rizzolo, and M.~Winkel.
\newblock {Interval partition diffusions: Connection with Petrov's
  Poisson--Dirichlet diffusions}.
\newblock Work in progress, revising parts of arXiv: 1609.06706v2 [math.PR],
  2019.

\bibitem{GnedPitm05}
A.~Gnedin and J.~Pitman.
\newblock Regenerative composition structures.
\newblock {\em Ann. Probab.}, 33(2):445--479, 2005.

\bibitem{GPY2006alpha}
A.~Gnedin, J.~Pitman, M.~Yor, et~al.
\newblock Asymptotic laws for compositions derived from transformed
  subordinators.
\newblock {\em The Annals of Probability}, 34(2):468--492, 2006.

\bibitem{Gnedin97}
A.~V. Gnedin.
\newblock The representation of composition structures.
\newblock {\em Ann. Probab.}, 25(3):1437--1450, 1997.

\bibitem{GnedinRegenerationSurvey}
A.~V. Gnedin.
\newblock Regeneration in random combinatorial structures.
\newblock {\em Probability Surveys}, 7:105--156, 2010.

\bibitem{JacodShiryaev}
J.~Jacod and A.~N. Shiryaev.
\newblock {\em Limit theorems for stochastic processes}, volume 288 of {\em
  Grundlehren der Mathematischen Wissenschaften [Fundamental Principles of
  Mathematical Sciences]}.
\newblock Springer-Verlag, Berlin, second edition, 2003.

\bibitem{Kingman1978}
J.~F.~C. Kingman.
\newblock Random partitions in population genetics.
\newblock {\em Proceedings of the Royal Society of London. A. Mathematical and
  Physical Sciences}, 361(1704):1--20, 1978.

\bibitem{LohrMytnWint18}
W.~L{\"o}hr, L.~Mytnik, and A.~Winter.
\newblock The {A}ldous chain on cladograms in the diffusion limit.
\newblock {\em arXiv preprint arXiv:1805.12057}, 2018.

\bibitem{M09}
G.~Miermont.
\newblock Tessellations of random maps of arbitrary genus.
\newblock {\em Ann. Sci. \'Ec. Norm. Sup\'er. (4)}, 42(5):725--781, 2009.

\bibitem{Petrov09}
L.~A. Petrov.
\newblock A two-parameter family of infinite-dimensional diffusions on the
  {K}ingman simplex.
\newblock {\em Funktsional. Anal. i Prilozhen.}, 43(4):45--66, 2009.

\bibitem{Pitman1995}
J.~Pitman.
\newblock Exchangeable and partially exchangeable random partitions.
\newblock {\em {Probab. Theory Related Fields}}, 102(2):145--158, 1995.

\bibitem{CSP}
J.~Pitman.
\newblock {\em Combinatorial stochastic processes}, volume 1875 of {\em Lecture
  Notes in Mathematics}.
\newblock Springer-Verlag, Berlin, 2006.
\newblock Lectures from the 32nd Summer School on Probability Theory held in
  Saint-Flour, July 7--24, 2002.

\bibitem{PitmWink09}
J.~Pitman and M.~Winkel.
\newblock Regenerative tree growth: binary self-similar continuum random trees
  and {P}oisson--{D}irichlet compositions.
\newblock {\em Ann. Probab.}, 37(5):1999--2041, 2009.

\bibitem{PitmYor92}
J.~Pitman and M.~Yor.
\newblock Arcsine laws and interval partitions derived from a stable
  subordinator.
\newblock {\em Proc. London Math. Soc. (3)}, 65(2):326--356, 1992.

\bibitem{RogersWilliams}
L.~C.~G. Rogers and D.~Williams.
\newblock {\em Diffusions, {M}arkov processes, and martingales. {V}ol. 1}.
\newblock Wiley Series in Probability and Mathematical Statistics: Probability
  and Mathematical Statistics. John Wiley \& Sons Ltd., Chichester, second
  edition, 1994.
\newblock Foundations.

\bibitem{RuggWalk09}
M.~Ruggiero and S.~G. Walker.
\newblock Countable representation for infinite dimensional diffusions derived
  from the two-parameter {P}oisson--{D}irichlet process.
\newblock {\em Electron. Commun. Probab.}, 14:501--517, 2009.

\bibitem{RuggWalkFava13}
M.~Ruggiero, S.~G. Walker, and S.~Favaro.
\newblock Alpha-diversity processes and normalized inverse-{G}aussian
  diffusions.
\newblock {\em Ann. Appl. Probab.}, 23(1):386--425, 2013.

\end{thebibliography}
